\documentclass{amsart}

\newtheorem{theorem}{Theorem}[section]
\newtheorem{lemma}[theorem]{Lemma}
\newtheorem{proposition}[theorem]{Proposition}

\theoremstyle{definition}

\theoremstyle{remark}
\newtheorem{remark}[theorem]{Remark}

\numberwithin{equation}{section}

\usepackage[T1]{fontenc}
\usepackage[utf8]{inputenc}
\usepackage{lmodern}

\usepackage{setspace}

\DeclareMathOperator{\tr}{tr}
\DeclareMathOperator{\divergent}{div}
\DeclareMathOperator{\Ric}{Ric}
\DeclareMathOperator{\real}{Re}
\DeclareMathOperator{\E}{\mathcal{E}}
\DeclareMathOperator{\M}{\mathcal{M}}
\DeclareMathOperator{\Vol}{Vol}
\DeclareMathOperator{\can}{can}

\newcommand{\R}{\mathbb{R}}
\newcommand{\ds}{\displaystyle}
\newcommand{\Y}{\mathcal{Y}}
\newcommand{\s}{\mathbb{S}}

\renewcommand{\div}{\divergent}
\renewcommand{\Re}{\real}

\begin{document}


\title[Rigidity of $\mbox{MOTSs}$ with Negative $\sigma$-constant]{Rigidity of Marginally Outer Trapped (Hyper)Surfaces with Negative $\sigma$-Constant}

\author{Abraão Mendes}
\address{Instituto de Matemática, Universidade Federal de Alagoas, Maceió, Alagoas, Brazil}
\email{abraao.rego@im.ufal.br}
\thanks{This work was carried out while the author was a Visiting Graduate Student at Princeton University during the 2015-2016 academic year. He was partially supported by NSF grant DMS-1104592 and by the CAPES Foundation, Ministry of Education of Brazil. He would like to express his gratitude to his Ph.D. advisors Fernando Codá Marques, at Princeton University, and Marcos Petrúcio Cavalcante, at UFAL}

\date{\today}

\begin{abstract}
In this paper we generalize the main result of \cite{GallowayMendes} in two different situations: in the first case for MOTSs of genus greater than one and, in the second case, for MOTSs of high dimension with negative $\sigma$-constant. In both cases we obtain a splitting result for the ambient manifold when it contains a stable closed MOTS which saturates a lower bound for the area (in dimension 2) or for the volume (in dimension $\ge3$). These results are extensions of \cite[Theorem 3]{Nunes} and \cite[Theorem 3]{Moraru} to general (non-time-symmetric) initial data sets.
\end{abstract}

\maketitle

\section{Introduction}

Let $M^3=(M^3,g)$ be a Riemannian 3-manifold. It is a very interesting question to know how the topology of a minimal surface $\Sigma^2\subset M^3$ can influence the geometry of $M^3$, and {\em vice-versa}.

Using the Gauss-Bonnet theorem, the Gauss equation, and the stability inequality for minimal surfaces, Schoen and Yau observed that if $M^3$ is oriented and has positive scalar curvature, then $M^3$ does not admit orientable closed stable minimal surfaces of positive genus. Furthermore, they proved the following result (see \cite{SchoenYau1979}).
 
\begin{theorem}[Schoen-Yau]\label{theo.SchoenYau}
Let $M^3=(M^3,g)$ be an oriented closed Riemannian 3-manifold with nonnegative scalar curvature. If the fundamental group of $M^3$ admits a subgroup abstractly isomorphic to the fundamental group of the 2-torus, then $M^3$ is flat.
\end{theorem} 

In order to prove this result, Schoen and Yau first proved that the hypothesis on the fundamental group of $M^3$ ensures the existence of a stable minimal 2-torus. After, they used that if $M$ admits a metric of nonnegative scalar curvature which is not flat, then $M$ also admits a metric of positive scalar curvature (see \cite{KazdanWarner}). Then, the result follows.

As observed by Fischer-Colbrie and Schoen \cite{Fischer-ColbrieSchoen}, if $\Sigma^2$ is an orientable closed stable minimal surface of genus $g(\Sigma)\ge1$ in an oriented Riemannian 3-manifold $M^3$ with scalar curvature $R\ge0$, then $\Sigma^2$ is a totally geodesic flat 2-torus in $M^3$, and $R=0$ on $\Sigma^2$. In the same work (see \cite[Remark 4]{Fischer-ColbrieSchoen}), Fischer-Colbrie and Schoen posed the problem of establishing a stronger (more global) rigidity statement if, say, the 2-torus is suitably area minimizing. 

Partially motivated by some issues concerning the topology of back holes, Cai and Galloway \cite{CaiGalloway2000} solved the problem posed by Fischer-Colbrie and Schoen, in a paper which has inspired a great deal of subsequent works in the subject (e.g. \cite{Ambrozio,BrayBrendleNeves,MicallefMoraru,Nunes}). They proved:

\begin{theorem}[Cai-Galloway]\label{theo.CaiGalloway}
Let $M^3$ be a Riemannian 3-manifold with nonnegative scalar curvature. If $\Sigma^2\subset M^3$ is a two-sided embedded 2-torus which is locally area minimizing, then a neighborhood of $\Sigma$ in $M$ is isometric to the product $((-\varepsilon,\varepsilon)\times\Sigma,dt^2+g_0)$, where $g_0$, the metric on $\Sigma$ induced from $M$, is flat. Moreover, if $M$ is complete and $\Sigma$ has least area in its isotopy class, $M$ is globally flat. 
\end{theorem}

As observed by Cai and Galloway, in the first part of the theorem above, $M$ need not be globally flat. Also, the second part remains true if $M$ is a manifold with boundary $\partial M\neq\emptyset$, assuming that its boundary is mean convex. A higher dimensional version of Theorem \ref{theo.CaiGalloway} was obtained in \cite{Cai}; see also \cite{Galloway2011} for a simplified proof.

Two similar results to Theorem \ref{theo.CaiGalloway} were obtained by Bray, Brendle, and Neves \cite{BrayBrendleNeves} and by Nunes \cite{Nunes} under different hypotheses on the scalar curvature of $M^3$ and the topology of $\Sigma^2$, which we paraphrase as follow.

\begin{theorem}[Bray-Brendle-Neves]\label{theo.BrayBrendleNeves}
Let $M^3$ be a Riemannian 3-manifold with scalar curvature bounded from below by $2c$, for some constant $c>0$. If $\Sigma^2$ is an embedded 2-sphere which is locally area minimizing, then the area of $\Sigma$ satisfies
\begin{eqnarray*}
A(\Sigma)\le\frac{4\pi}{c}.
\end{eqnarray*}
Furthermore, if equality holds, a neighborhood of $\Sigma$ in $M$ is isometric to the product $((-\varepsilon,\varepsilon)\times\Sigma,dt^2+g_0)$, where $(\Sigma^2,g_0)$ is the round 2-sphere of Gaussian curvature $\kappa=c$. In this case, if further $M$ is complete and $\Sigma$ has least area in its isotopy class, the universal cover of $M$ is isometric to the product $(\R\times\Sigma,dt^2+g_0)$.
\end{theorem}

\begin{theorem}[Nunes]\label{theo.Nunes}
Let $M^3$ be a Riemannian 3-manifold with scalar curvature bounded from below by $-2c$, for some constant $c>0$. If $\Sigma^2\subset M^3$ is a two-sided embedded closed Riemann surface of genus $g(\Sigma)\ge2$ which is locally area minimizing, then the area of $\Sigma$ satisfies 
\begin{eqnarray*}
A(\Sigma)\ge\frac{4\pi(g(\Sigma)-1)}{c}.
\end{eqnarray*}
Furthermore, if equality holds, a neighborhood of $\Sigma$ in $M$ is isometric to the product $((-\varepsilon,\varepsilon)\times\Sigma,dt^2+g_0)$, where $(\Sigma^2,g_0)$ has constant Gaussian curvature $\kappa=-c$. In this case, if further $M$ is complete and $\Sigma$ has least area in its isotopy class, the universal cover of $M$ is isometric to the product $(\R\times\Sigma,dt^2+g_0)$.
\end{theorem}

The original proofs of Theorems \ref{theo.CaiGalloway}, \ref{theo.BrayBrendleNeves}, and \ref{theo.Nunes} are very different. However, Micallef and Moraru \cite{MicallefMoraru} presented a unified proof for them. 

From the point of view of relativity, Theorems \ref{theo.CaiGalloway}, \ref{theo.BrayBrendleNeves}, and \ref{theo.Nunes} may be viewed as statements about time-symmetric (totally geodesic) initial data sets. In \cite{GallowayMendes}, Galloway and the author generalized Theorem \ref{theo.BrayBrendleNeves} to general (non-time-symmetric) initial data sets. In that more general situation, minimal surfaces were replaced by {\em marginally outer trapped surfaces} (MOTSs). We proved (see Section \ref{sec.Preliminaries} for definitions):

\begin{theorem}[Galloway-Mendes]\label{theo.GallowayMendes}
Let $M^3=(M^3,g,K)$ be a 3-dimensional initial data set in a spacetime $\bar M^4=(\bar M^4,\bar g)$. Let $\Sigma^2$ be a spherical MOTS in $M^3$ which is weakly outermost and outer area minimizing. Suppose that $\mu-|J|\ge c$ on $M_+$ for some constant $c>0$. Then, 
the area of $\Sigma$ satisfies
\begin{eqnarray*}
A(\Sigma)\le\frac{4\pi}{c}.
\end{eqnarray*}
Furthermore, if equality holds, we have:
\begin{enumerate}
\item An outer neighborhood $U\approx[0,\varepsilon)\times\Sigma$ of $\Sigma$ in $M$ is isometric to 
\begin{eqnarray*}
([0,\varepsilon)\times\Sigma,dt^2+g_0),
\end{eqnarray*}
where $(\Sigma^2,g_0)$ is the round 2-sphere of Gaussian curvature $\kappa=c$.
\item Each slice $\Sigma_t\approx\{t\}\times\Sigma$ is totally geodesic as a submanifold of spacetime. Equivalently, $\chi_+(t)=\chi_-(t)=0$, where $\chi_\pm(t)$ are the null second fundamental forms of $\Sigma_t$ in $\bar M^4$.
\item $K(\cdot,\cdot)|_{T_x\Sigma_t}=0$ and $K(\nu_t,\cdot)|_{T_x\Sigma_t}=0$ for each $x\in\Sigma_t$, where $\nu_t$ is the outer unit normal to $\Sigma_t$, and $J=0$ on $U$.
\end{enumerate}
\end{theorem}

Above, $\mu$ and $J$ are defined in terms of the Einstein tensor of $\bar M^4$, $G=\bar\Ric-\frac{\bar R}{2}\bar g$, by $\mu=G(u,u)$ and $J(\cdot)=G(u,\cdot)|_{T_pM}$, $p\in M$, where $u$ is the future directed timelike unit normal to $M$.

An important generalization of Theorem \ref{theo.Nunes} is due to Moraru \cite{Moraru}, who extended Nunes' result to closed hypersurfaces $\Sigma^n$ of dimension $n\ge3$. He proved:

\begin{theorem}[Moraru]\label{theo.Moraru}
Let $M^{n+1}$ be a Riemannian manifold of dimension $n+1$, $n\ge3$, with scalar curvature bounded from below by $-2c$, for some constant $c>0$. If $\Sigma^n\subset M^{n+1}$ is a two-sided embedded closed hypersurface with $\sigma(\Sigma)<0$ which is locally volume minimizing, then the volume of $\Sigma$ satisfies 
\begin{eqnarray*}
\Vol(\Sigma^n)\ge\left(\frac{|\sigma(\Sigma^n)|}{2c}\right)^{\frac{n}{2}}.
\end{eqnarray*}
Furthermore, if equality holds, a neighborhood of $\Sigma$ in $M$ is isometric to the product $((-\varepsilon,\varepsilon)\times\Sigma,dt^2+g_\Sigma)$, where $g_\Sigma$, the metric on $\Sigma$ induced from $M$, is Einstein with constant scalar curvature $R_\Sigma=-2c$.
\end{theorem}

Above, $\sigma(\Sigma)$ is the topological invariant introduced by Schoen \cite{Schoen1989}, called the {\em $\sigma$-constant} of $\Sigma$ (see also Kobayashi \cite{Kobayashi}).

In the same spirit of Theorem \ref{theo.GallowayMendes}, in the present paper we generalize Theorems \ref{theo.Nunes} and \ref{theo.Moraru} to non-time-symmetric initial data sets.

Our first result is the following.

\begin{theorem}[Theorem \ref{theo.Split.Dimension.3}]\label{theo.1.7}
Let $M^3=(M^3,g,K)$ be a 3-dimensional initial data set. Let $\Sigma^2$ be an orientable weakly outermost closed MOTS in $M^3$ of genus $g(\Sigma)\ge2$. Suppose that $\mu-|J|\ge-c$ for some constant $c>0$ and that $K$ is 2-convex, both on $M_+$. Then, the area of $\Sigma$ satisfies
\begin{eqnarray}\label{eq.GibbonsWoolgar}
A(\Sigma)\ge\frac{4\pi(g(\Sigma)-1)}{c}.
\end{eqnarray}
Furthermore, if equality holds, we have:
\begin{enumerate}
\item An outer neighborhood $U\approx[0,\varepsilon)\times\Sigma$ of $\Sigma$ in $M$ is isometric to 
\begin{eqnarray*}
([0,\varepsilon)\times\Sigma,dt^2+g_0),
\end{eqnarray*}
where $(\Sigma,g_0)$ has constant Gaussian curvature $\kappa=-c$.
\item Each slice $\Sigma_t\approx\{t\}\times\Sigma$ is totally geodesic as a submanifold of spacetime. Equivalently, $\chi_+(t)=\chi_-(t)=0$, where $\chi_\pm(t)$ are the null second fundamental forms of $\Sigma_t$.
\item $K(\cdot,\cdot)|_{T_x\Sigma_t}=0$ and $K(\nu_t,\cdot)|_{T_x\Sigma_t}=0$ for each $x\in\Sigma_t$, where $\nu_t$ is the outer unit normal to $\Sigma_t$, and $J=0$ on $U$.
\end{enumerate}
\end{theorem}

We point out that the area estimate (\ref{eq.GibbonsWoolgar}) was proved by Gibbons \cite{Gibbons} in the time-symmetric case and by Woolgar \cite{Woolgar} in the general case.

Our second result is the following.

\begin{theorem}[Theorem \ref{theo.Negative.sigma.Constant}]\label{theo.1.8}
Let $M^{n+1}=(M^{n+1},g,K)$ be a $(n+1)$-dimensional initial data set, $n\ge3$. Let $\Sigma^n$ be a weakly outermost closed MOTS in $M^{n+1}$ with $\sigma$-constant $\sigma(\Sigma)<0$. Suppose that $\mu-|J|\ge-c$ for some constant $c>0$ and that $K$ is $n$-convex, both on $M_+$. Then, the volume of $\Sigma$ satisfies
\begin{eqnarray}\label{eq.GallowayMurchadha}
\Vol(\Sigma^n)\ge\left(\frac{|\sigma(\Sigma^n)|}{2c}\right)^{\frac{n}{2}}.
\end{eqnarray}
Furthermore, if equality holds, we have:
\begin{enumerate}
\item An outer neighborhood $U\approx[0,\varepsilon)\times\Sigma$ of $\Sigma$ in $M$ is isometric to 
\begin{eqnarray*}
([0,\varepsilon)\times\Sigma,dt^2+g_\Sigma),
\end{eqnarray*}
and $(\Sigma,g_\Sigma)$ is Einstein with constant scalar curvature $R_\Sigma=-2c$.
\item Each slice $\Sigma_t\approx\{t\}\times\Sigma$ is totally geodesic as a submanifold of spacetime. Equivalently, $\chi_+(t)=\chi_-(t)=0$, where $\chi_\pm(t)$ are the null second fundamental forms of $\Sigma_t$.
\item $K(\cdot,\cdot)|_{T_x\Sigma_t}=0$ and $K(\nu_t,\cdot)|_{T_x\Sigma_t}=0$ for each $x\in\Sigma_t$, where $\nu_t$ is the outer unit normal to $\Sigma_t$, and $J=0$ on $U$.
\end{enumerate}
\end{theorem}

The volume estimate (\ref{eq.GallowayMurchadha}) was obtained by Cai and Galloway \cite{CaiGalloway2001} in the time-symmetric case and by Galloway and Murchadha \cite{GallowayMurchadha} in the general case. 

Although we are assuming a convexity hypothesis on the initial data set $M$ in Theorems \ref{theo.1.7} and \ref{theo.1.8}, we are not assuming any minimizing hypothesis on the area (in dimension 2) nor the volume (in dimension $\ge3$) of $\Sigma$.

Since in time-symmetric initial data sets $\mu-|J|$ reduces to the half of the scalar curvature, the hypothesis $\mu-|J|\ge-c$ in general inicial data sets is the analogue of $R\ge-2c$ in time-symmetric ones.

\section{Preliminaries}\label{sec.Preliminaries}

Let $\bar M^{n+2}=(\bar M^{n+2},\bar g,\bar\nabla)$ be a spacetime of dimension $n+2$, i.e., a connected, oriented and time-oriented Lorentzian $(n+2)$-manifold, where $\bar g$ and $\bar\nabla$ denote a Lorentzian metric and its Levi-Civita connection, respectively. Let $M^{n+1}$ be a spacelike hypersurface embedded into $\bar M^{n+2}$ and let $u$ denote the future directed timelike unit normal to $M$. In this case, we can see $M^{n+1}$ as an initial data set $M^{n+1}=(M^{n+1},g,K)$, where $g$ is the Riemannian metric on $M$ induced from $\bar M$ and $K$ is the second fundamental form of $M$ in $\bar M$. More precisely, 
\begin{eqnarray}\label{eq.chi_pm}
K(V,W)=\bar g(\bar\nabla_Vu,W),\,\,\,\mbox{for all}\,\,\,V,W\in T_pM\,\,\,\mbox{and}\,\,\,p\in M.
\end{eqnarray}

Now, let $\Sigma^n$ be a two-sided connected closed (compact with no boundary) hypersurface embedded into $M^{n+1}$. Fix a unit normal to $\Sigma$ in $M$, say $\nu$, that is, $\nu(x)\in T_x\Sigma^\perp\cap T_xM$ for all $x\in\Sigma$. By convention, we say that $\nu$ is outward pointing and, in turn, $-\nu$ is inward pointing. Then, we can define two future directed null normal vector fields along $\Sigma$, $l_+=u+\nu$ and $l_-=u-\nu$, which are outward and inward pointing, respectively. The {\em null second fundamental forms} of $\Sigma^n$ in $\bar M^{n+2}$, $\chi_+$ and $\chi_-$, are defined by
\begin{eqnarray*}
\chi_\pm(Y,Z)=K(Y,Z)\pm A(Y,Z),\,\,\,\mbox{for all}\,\,\,Y,Z\in T_x\Sigma\,\,\,\mbox{and}\,\,\,x\in\Sigma,
\end{eqnarray*}
where $A$ is the second fundamental form of $\Sigma$ in $M$. On the other hand, the {\em null mean curvatures} of $\Sigma$ in $\bar M$, $\theta_+$ and $\theta_-$, are defined by taking the traces of $\chi_+$ and $\chi_-$ with respect to $\Sigma$, 
\begin{eqnarray*}
\theta_\pm=\tr_\Sigma\chi_\pm.
\end{eqnarray*} 
Observe that 
\begin{eqnarray*}
\theta_\pm=\div_\Sigma(l_\pm)=\tr_\Sigma K\pm H,
\end{eqnarray*}
where $H=\tr A$ is the mean curvature of $\Sigma$ in $M$. 

As introduced by Penrose, $\Sigma$ is said to be a {\em trapped surface} if both $\theta_+$ and $\theta_-$ are negative everywhere. Also, $\Sigma$ is an {\em outer trapped surface} if $\theta_+$ is negative (with no assumption on $\theta_-$). Finally, we say that $\Sigma$ is a {\em marginally outer trapped surface} (MOTS) if $\theta_+$ vanishes identically (with no assumption on $\theta_-$). For simplicity, we drop the plus sign of $l_+$, $\chi_+$, and $\theta_+$.

\begin{remark}
Observe that in the time-symmetric case, i.e., when $K\equiv 0$, $\chi$ and $\theta$ reduce to $A$ and $H$, respectively. Then, in this case, a MOTS is just a minimal hypersurface.
\end{remark}

Before presenting the notion of stability for MOTSs introduced by Andersson, Mars, and Simon \cite{AnderssonMarsSimon2005,AnderssonMarsSimon2008}, we are going to fix some notations:
\begin{itemize}
\item Let $G$ be the Einstein tensor of $\bar M=(\bar M,\bar g,\bar\nabla)$ given by
\begin{eqnarray*}
G=\bar\Ric-\frac{\bar R}{2}\bar g,
\end{eqnarray*}
where $\bar\Ric$ and $\bar R$ are the Ricci tensor and the scalar curvature of $\bar M$, respectively;
\item Denote $G(u,u)$ by $\mu$ on $M$ and $G(u,\cdot)$ by $J(\cdot)$ on $T_pM$ for each $p\in M$;
\item Define
\begin{eqnarray*}
Q=\frac{1}{2}R_\Sigma-(\mu+J(\nu))-\frac{1}{2}|\chi|^2,
\end{eqnarray*} 
where $R_\Sigma$ is the scalar curvature of $\Sigma$ with respect to the Riemannian metric $\langle\,,\,\rangle$ induced from $M=(M,g)$;
\item Denote by $X\in\Gamma(T\Sigma)$ the vector field on $\Sigma$ dual to $K(\nu(x),\cdot)|_{T_x\Sigma}$ for each $x\in\Sigma$;
\item Denote by $\tau$ the mean curvature of $M$ in $\bar M$, i.e., $\tau=\tr K$.
\end{itemize}

Now, consider $t\longmapsto\Sigma_t\subset M$, $t\in(-\varepsilon,\varepsilon)$, a variation of $\Sigma_0=\Sigma$ in $M$ with variation vector field $\frac{\partial}{\partial t}|_{t=0}=\phi\nu$, $\phi\in C^\infty(\Sigma)$. Denote by $\theta(t)$ the null mean curvature of $\Sigma_t$ in $\bar M$ with respect to $l_t=u+\nu_t$, that is, $\theta(t)=\div_{\Sigma_t}(l_t)$, where $\nu_t$ is the unit normal to $\Sigma_t$ in $M$ such that $\nu_0=\nu$. A similar calculation to that one in \cite{AnderssonMarsSimon2008} gives,
\begin{eqnarray}\label{eq.First.Variation.Theta}
\frac{\partial\theta}{\partial t}\bigg|_{t=0}=L\phi+\left(\theta(0)\tau-\frac{1}{2}\theta(0)^2\right)\phi,
\end{eqnarray} 
where $L:C^\infty(\Sigma)\to C^\infty(\Sigma)$ is the operator given by 
\begin{eqnarray}\label{eq.operator.L}
L\phi=-\Delta\phi+2\langle X,\nabla\phi\rangle+(Q-|X|^2+\div X)\phi.
\end{eqnarray}
Here, $\Delta$ and $\nabla$ denote the Laplacian and the gradient operators on $\Sigma=(\Sigma,\langle\,,\,\rangle)$, respectively. The operator $L$ is known as the MOTSs stability operator.

\begin{remark}
Using the Gauss-Codazzi equations, we can express $\mu$ and $J$ solely in terms of the initial data set $M=(M,g,K)$:
\begin{eqnarray*}
\mu&=&\frac{1}{2}(R+\tau^2-|K|^2),\\
J&=&\div K-d\tau,
\end{eqnarray*} 
where $R$ is the scalar curvature of $M$. Then, using the Gauss equation, 
\begin{eqnarray*}
\Ric(\nu,\nu)+|A|^2=\frac{1}{2}(R-R_\Sigma+|A|^2+H^2),
\end{eqnarray*}
where $\Ric$ is the Ricci tensor of $M$, we can see that in the time-symmetric case the operator $L$ reduces to the classical stability operator for minimal hypersurfaces,
\begin{eqnarray*}
-\Delta-(\Ric(\nu,\nu)+|A|^2),
\end{eqnarray*}
when $\Sigma$ is a MOTS (in this case, a minimal hypersurface).
\end{remark}

The operator $L$ is not self-adjoint in general. But, it has the following properties (see \cite{AnderssonMarsSimon2008} and the references therein).

\begin{lemma}\label{lemma.AnderssonMarsSimon}
The following holds for the operator $L$.
\begin{enumerate}
\item There is a real eigenvalue $\lambda_1=\lambda_1(L)$, called the {\em principal eigenvalue} of $L$, such that for any other eigenvalue $\mu$, $\Re(\mu)\ge\lambda_1$. The associated eigenfunction $\phi$, $L\phi=\lambda_1\phi$, is unique up to a multiplicative constant, and can be chosen to be strictly positive.
\item $\lambda_1\ge0$ (resp., $\lambda_1>0$) if and only if there exists $\psi\in C^\infty(\Sigma)$, $\psi>0$, such that $L\psi\ge0$ (resp., $L\psi>0$).  
\end{enumerate}
\end{lemma}

The lemma above is true for any operator of the form $\phi\longmapsto-\Delta\phi+\langle Y,\nabla\phi\rangle+q\phi$, where $Y\in\Gamma(T\Sigma)$ and $q\in C^\infty(\Sigma)$.

A closed MOTS $\Sigma$ is said to be {\em stable} if $\lambda_1(L)\ge0$. It follows from the lemma above that $\Sigma$ is stable if and only if there exists $\psi\in C^\infty(\Sigma)$, $\psi>0$, such that $L\psi\ge0$. Observe that in the time-symmetric case, $\Sigma$ is stable as a MOTS if and only if $\Sigma$ is stable as a minimal hypersurface, since, in this case, $\lambda_1(L)$ is the first eigenvalue of the Jacobi operator $L$.

Now, consider the ``symmetrized'' operator $L_0:C^\infty(\Sigma)\to C^\infty(\Sigma)$,
\begin{eqnarray*}
L_0\phi=-\Delta\phi+Q\phi,
\end{eqnarray*}
obtained formally from (\ref{eq.operator.L}) by taking $X=0$. As first observed by Galloway \cite{Galloway2008}, using the Rayleigh's formula for the first eigenvalue of $L_0$ (which coincides with the principal eigenvalue of $L_0$),
\begin{eqnarray}\label{eq.RayleighL0}
\lambda_1(L_0)=\inf_{f\in C^\infty(\Sigma)\setminus\{0\}}\frac{\int_\Sigma(|\nabla f|^2+Qf^2)dA}{\int_\Sigma f^2dA},
\end{eqnarray}
and the main argument presented in \cite{GallowaySchoen}, we have the following result.
 
\begin{lemma}\label{lemma.GallowaySchoen}
$\lambda_1(L_0)\ge\lambda_1(L)$. In particular, if $\lambda_1(L)\ge0$, 
\begin{eqnarray}\label{eq.Stability.L0}
\int_\Sigma(|\nabla f|^2+Qf^2)dA\ge0
\end{eqnarray}
for all $f\in C^\infty(\Sigma)$.
\end{lemma}

The next result is a very useful consequence of the proof of the main theorem proved in \cite{Galloway2008} (see also \cite{GallowayMendes}).

\begin{lemma}\label{lemma.Folliation}
Let $\Sigma^n$ be a closed MOTS in an initial data set $M^{n+1}=(M^{n+1},g,K)$. If $\lambda_1(L)=0$, then there exists a variation $t\longmapsto\Sigma_t$ of $\Sigma_0=\Sigma$ in $M$, $t\in(-\varepsilon,\varepsilon)$, such that $\Sigma_t$ is a hypersurface with constant null mean curvature $\theta=\theta(t)$ for each $t\in(-\varepsilon,\varepsilon)$. Furthermore, $\{\Sigma_t\}_{t\in(-\varepsilon,\varepsilon)}$ is a foliation of a neighborhood of $\Sigma$ in $M$ with variation vector field $\frac{\partial}{\partial t}=\phi_t\nu_t$, where $\nu_t$ is the outer unit normal to $\Sigma_t$ and $\phi_t:\Sigma_t\approx\{t\}\times\Sigma\to\R$ is a positive function.
\end{lemma}

Before passing to the next section, let us fix some notation and terminologies. 

If $\Sigma^n$ is a separating MOTS in $M^{n+1}=(M^{n+1},g,K)$, let $M_+$ be the region consisting of $\Sigma$ and the region outside of $\Sigma$. We say that $\Sigma$ is {\em outermost} if there are no outer trapped or marginally outer trapped surfaces in $M_+$ homologous to $\Sigma$. We say that $\Sigma$ is {\em weakly outermost} if there are no outer trapped surfaces in $M_+$ homologous to $\Sigma$. Also, $\Sigma$ is said to be {\em outer area minimizing} if its area is less than or equal to the area of any surface in $M_+$ homologous to $\Sigma$.

\begin{remark}\label{remark.WeaklyOutermostStable}
It follows from (\ref{eq.First.Variation.Theta}) that if $\Sigma$ is a weakly outermost MOTS, then it is stable. Otherwise, taking a variation $t\longmapsto\Sigma_t$ such that $\frac{\partial}{\partial t}|_{t=0}=\phi\nu$, where $\phi>0$ is the eigenfunction associated to $\lambda_1=\lambda_1(L)$, we have that $\Sigma_t$ is an outer trapped surface for $t>0$ small enough, which is homologous to $\Sigma$ in $M_+$. 
\end{remark}

\section{The Case of High Genus}\label{sec.HighGenus}

In this section we generalize the main result obtained in \cite{GallowayMendes} to surfaces $\Sigma$ of genus $g(\Sigma)\ge2$.

Inequality (\ref{eq.AreaEstimates}) below is a well known area estimate (see \cite{Gibbons} for the time-symmetric case and \cite{Woolgar} for the general case). But, for the sake of completeness, we will present its proof and some of the consequences when equality occurs. 

\begin{proposition}\label{prop.GibbonsWoolgar}
Let $\Sigma^2$ be an orientable stable closed MOTS of genus $g(\Sigma)\ge2$ in a 3-dimensional initial data set $M^3=(M^3,g,K)$. Suppose that $\mu+J(\nu)\ge-c$ on $\Sigma$ for some constant $c>0$. Then, the area of $\Sigma$ satisfies
\begin{eqnarray}\label{eq.AreaEstimates}
A(\Sigma)\ge\frac{4\pi(g(\Sigma)-1)}{c}.
\end{eqnarray}   
Furthermore, if equality holds, $\Sigma$ has constant Gaussian curvature $\kappa=-c$, the null second fundamental form $\chi=\chi_+$ of $\Sigma$ vanishes identically, $\mu+J(\nu)=-c$ on $\Sigma$, and $\lambda_1(L_0)=\lambda_1(L)=0$.
\end{proposition}

\begin{proof}
By hypothesis, $\lambda_1(L)\ge0$, where $L$ is the MOTSs stability operator. Then, using the Gauss-Bonnet theorem and inequality (\ref{eq.Stability.L0}) for $f=1$, we have
\begin{eqnarray*}
0&\le&\int_\Sigma QdA\hspace{.2cm}=\hspace{.2cm}\int_\Sigma\left(\kappa-(\mu+J(\nu))-\frac{1}{2}|\chi|^2\right)dA\\
&\le&\int_\Sigma\left(\kappa+c-\frac{1}{2}|\chi|^2\right)dA\hspace{.2cm}\le\hspace{.2cm}\int_\Sigma (\kappa+c)dA\\
&=&4\pi(1-g(\Sigma))+cA(\Sigma),
\end{eqnarray*}
which proves (\ref{eq.AreaEstimates}).
 
Now, if equality in (\ref{eq.AreaEstimates}) holds, all inequalities above must be equalities. Then, $\chi=0$, $\mu+J(\nu)=-c$ on $\Sigma$, and $\int_\Sigma QdA=0$. Observe that $Q=\kappa+c$. Using (\ref{eq.Stability.L0}) again, for all $\alpha\in\mathbb R$ and $\psi\in C^\infty(\Sigma)$, we have
\begin{eqnarray*}
0&\le&\int_\Sigma(|\nabla(\alpha+\psi)|^2+Q(\alpha+\psi)^2)dA\\
&=&\int_\Sigma(|\nabla\psi|^2+Q\psi^2)dA+2\alpha\int_\Sigma Q\psi dA.
\end{eqnarray*}
This implies $\int_\Sigma Q\psi dA=0$ for all $\psi\in C^\infty(\Sigma)$, hence $Q=\kappa+c=0$. Finally, $\lambda_1(L_0)=\lambda_1(L)=0$ follows from Lemma \ref{lemma.GallowaySchoen} and Rayleigh's formula (\ref{eq.RayleighL0}), remembering that $\lambda_1(L)\ge0$.
\end{proof}

Before passing to the main result of this section, we are going to prove a calculus lemma and to present the definition of $n$-convexity. The proof of the following lemma is based on the techniques presented in \cite{MicallefMoraru}.

\begin{lemma}\label{lemma.Calculus.Lemma}
Let $f\in C^1([0,\varepsilon))$ and $\eta,\xi,\rho\in C^0([0,\varepsilon))$ be functions such that $\max\{f,\rho\}\ge0$, $\xi\ge0$, $\eta>0$, $f(0)=0$, and
\begin{eqnarray*}
f'(t)\eta(t)\le\int_0^tf(s)\xi(s)ds+f(t)\rho(t),\,\,\,\forall t\in[0,\varepsilon).
\end{eqnarray*}
Then, $f\le0$ everywhere. In particular, if $f\ge0$, it must be identically zero.
\end{lemma}

\begin{proof}
Define
\begin{eqnarray*}
I=\{\delta\in(0,\varepsilon);f\le0\,\,\,\mbox{in}\,\,\,[0,\delta]\}.
\end{eqnarray*}
{\bf Claim 1.} $I\neq\emptyset$.\\
{\em Proof of Claim 1.} 
By continuity, there exists a constant $C>0$ satisfying
\begin{eqnarray*}
\max_{t\in[0,\varepsilon/2]}\left\{\frac{1}{\eta(t)}\int_0^t\xi(s)ds,\frac{\rho(t)}{\eta(t)}\right\}\le C.
\end{eqnarray*}
Choose $\delta\in(0,\varepsilon/2]$ such that $0<1-2C\delta$. We claim that $\delta\in I$. Otherwise, fix $t_0\in(0,\delta]$ with $f(t_0)>0$ and define 
\begin{eqnarray*}
t_1=\inf\{t\in[0,t_0];f(t)\ge f(t_0)\}.
\end{eqnarray*}
By continuity, $f(t_1)=f(t_0)>0$. Moreover, by definition, $f(t)\le f(t_1)$ for all $t\in[0,t_1]$ (observe that $t_1>0$, 
since $f(0)=0$). Then, by the mean value theorem, there exists $t^*\in(0,t_1)$ such that $f(t_1)=f'(t^*)t_1$. Therefore,
\begin{eqnarray*}
\frac{f(t_1)}{t_1}&=&f'(t^*)\hspace{.2cm}\le\hspace{.2cm}\frac{1}{\eta(t^*)}\int_0^{t^*}f(s)\xi(s)ds+f(t^*)\frac{\rho(t^*)}{\eta(t^*)}\\
&\le&\frac{f(t_1)}{\eta(t^*)}\int_0^{t^*}\xi(s)ds+f(t^*)\frac{\rho(t^*)}{\eta(t^*)}\\
&\le&Cf(t_1)+f(t^*)\frac{\rho(t^*)}{\eta(t^*)}.
\end{eqnarray*}
Now, observe that $f(t^*)\rho(t^*)/\eta(t^*)\le Cf(t_1)$. In fact, if $f(t^*)\rho(t^*)\le0$, we have done, since $f(t_1)>0$. If $f(t^*)\rho(t^*)>0$, then $f(t^*)>0$ and $\rho(t^*)>0$, because $\max\{f(t^*),\rho(t^*)\}\ge0$ by hypothesis, which implies that
\begin{eqnarray*}
f(t^*)(\rho(t^*)/\eta(t^*))\le f(t^*)C\le f(t_1)C.
\end{eqnarray*}
In any event,
\begin{eqnarray*}
\frac{f(t_1)}{t_1}&\le&2Cf(t_1),
\end{eqnarray*}
hence $1\le 2Ct_1\le2C\delta$, which is a contradiction, because $0<1-2C\delta$. This finishes the proof of Claim 1.\\

\hspace{-0.45cm}{\bf Claim 2.} $\sup I=\varepsilon$.\\
{\em Proof of Claim 2.} Define $\delta_0=\sup I\in(0,\varepsilon]$. If $\delta_0<\varepsilon$, set $\tilde f(t)=f(t+\delta_0)$, $\tilde\eta(t)=\eta(t+\delta_0)$, $\tilde\xi(t)=\xi(t+\delta_0)$, and $\tilde\rho(t)=\rho(t+\delta_0)$ for $t\in[0,\varepsilon-\delta_0)$. Observing that, in this case, $f\le0$ in $[0,\delta_0]$ (i.e., $\delta_0\in I$), we have
\begin{eqnarray*}
\tilde f'(t)\tilde\eta(t)&=&f'(t+\delta_0)\eta(t+\delta_0)\\
&\le&\int_0^{t+\delta_0}f(s)\xi(s)ds+f(t+\delta_0)\rho(t+\delta_0)\\
&\le&\int_0^t\tilde f(r)\tilde\xi(r)dr+\tilde f(t)\tilde\rho(t),
\end{eqnarray*}
for all $t\in[0,\varepsilon-\delta_0)$. From the definition of $\delta_0$, $\tilde f(0)=f(\delta_0)=0$. Then, using Claim 1 for these new functions, there exists $\delta_1\in(0,\varepsilon-\delta_0)$ such that $\tilde f\le0$ in $[0,\delta_1]$, which implies $f\le0$ in $[0,\delta_0+\delta_1]$, contradicting the definition of $\delta_0$. This finishes the proof of Claim 2.
 
The result follows from Claim 2.
\end{proof}

Let $M^{n+1}=(M^{n+1},g,K)$ be an initial data set of dimension $n+1$. We say that $K$ is {\em $n$-convex} (on $M$) if $\tr_\pi K\ge0$ for all $\pi\subset T_pM$ and $p\in M$, where $\pi$ is a linear subspace of dimension $n$. Saying that $K$ is $n$-convex is equivalent to say that the sum of the $n$ smallest eigenvalues of $K$ is nonnegative.

The main result of this section is the following.

\begin{theorem}\label{theo.Split.Dimension.3}
Let $M^3=(M^3,g,K)$ be a 3-dimensional initial data set. Let $\Sigma^2$ be an orientable weakly outermost closed MOTS in $M^3$ of genus $g(\Sigma)\ge2$. Suppose that $\mu-|J|\ge-c$ for some constant $c>0$ and that $K$ is 2-convex, both on $M_+$. Then, if $A(\Sigma)=4\pi(g(\Sigma)-1)/c$, the following hold.
\begin{enumerate}
\item An outer neighborhood $U\approx[0,\varepsilon)\times\Sigma$ of $\Sigma$ in $M$ is isometric to 
\begin{eqnarray*}
([0,\varepsilon)\times\Sigma,dt^2+g_0),
\end{eqnarray*}
where $(\Sigma,g_0)$ has constant Gaussian curvature $\kappa=-c$.
\item Each slice $\Sigma_t\approx\{t\}\times\Sigma$ is totally geodesic as a submanifold of spacetime. Equivalently, $\chi_+(t)=\chi_-(t)=0$, where $\chi_\pm(t)$ are the null second fundamental forms of $\Sigma_t$.
\item $K(\cdot,\cdot)|_{T_x\Sigma_t}=0$ and $K(\nu_t,\cdot)|_{T_x\Sigma_t}=0$ for each $x\in\Sigma_t$, where $\nu_t$ is the outer unit normal to $\Sigma_t$, and $J=0$ on $U$.
\end{enumerate}
\end{theorem}

\begin{proof}
First, observe that $\Sigma$ is stable, because it is weakly outermost. Then, by Proposition \ref{prop.GibbonsWoolgar}, we have $\lambda_1(L)=0$. In this case, let $\{\Sigma_t\}_{t\in(-\varepsilon,\varepsilon)}$, $\theta=\theta(t)$, and $\phi=\phi_t$ be given by Lemma \ref{lemma.Folliation}. It follows from (\ref{eq.First.Variation.Theta}) that 
\begin{eqnarray*}
\frac{d\theta}{dt}=-\Delta\phi+2\langle X,\nabla\phi\rangle+\left(Q-|X|^2+\div X-\dfrac{1}{2}\theta^2+\theta\tau\right)\phi,
\end{eqnarray*}
where $\Delta=\Delta_t$, $\langle\,,\,\rangle=\langle\,,\,\rangle_t$, $X=X_t$, $\nabla=\nabla_t$ (gradient operator), $Q=Q_t$, and $\div=\div_t$ are the respective entities associated to $\Sigma_t$, for each $t\in(-\varepsilon,\varepsilon)$. Thus,
\begin{eqnarray*}
\frac{\theta'}{\phi}&=&-\frac{\Delta\phi}{\phi}+2\langle X,\frac{1}{\phi}\nabla\phi\rangle-|X|^2+\div X+Q-\frac{1}{2}\theta^2+\theta\tau\\
&=&\div Y-|Y|^2+Q-\frac{1}{2}\theta^2+\theta\tau\\
&\le&\div Y+Q+\theta\tau,
\end{eqnarray*} 
where $Y=X-\nabla\ln\phi$. Therefore, observing that $\theta'(t)$ is also constant on $\Sigma_t$ and using the divergence theorem, for each $t\in[0,\varepsilon)$, we have
\begin{eqnarray*}
\theta'(t)\int_{\Sigma_t}\frac{1}{\phi}dA_t-\theta(t)\int_{\Sigma_t}\tau dA_t&\le&\int_{\Sigma_t}QdA_t\\
&=&\int_{\Sigma_t}\left(\kappa-(\mu+J(\nu))-\frac{1}{2}|\chi|^2\right)dA_t\\
&\le&\int_{\Sigma_t}\left(\kappa-(\mu-|J|)-\frac{1}{2}|\chi|^2\right)dA_t\\
&\le&\int_{\Sigma_t}(\kappa-(\mu-|J|))dA_t.
\end{eqnarray*}
Now, using the last inequality above along with the hypothesis that $\mu-|J|\ge-c$ on $U$, we get
\begin{eqnarray*}
\theta'(t)\int_{\Sigma_t}\frac{1}{\phi}dA_t-\theta(t)\int_{\Sigma_t}\tau dA_t&\le&\int_{\Sigma_t}(\kappa+c)dA_t\\
&=&4\pi(1-g(\Sigma))+cA(\Sigma_t)\\
&=&-cA(\Sigma_0)+cA(\Sigma_t)\\
&=&c\int_0^t\frac{d}{ds}A(\Sigma_s)ds.
\end{eqnarray*}
Above we have used the Gauss-Bonnet theorem and the fundamental theorem of calculus. On the other side, $\theta(t)=\tr_{\Sigma_t}K+H_t\ge H_t$ for $t\in[0,\varepsilon)$, since $K$ is 2-convex on $U$. Here, $H_t$ is the mean curvature of $\Sigma_t$ in $M$. Then, by the first variation of area, 
\begin{eqnarray*}
\theta'(t)\int_{\Sigma_t}\frac{1}{\phi}dA_t-\theta(t)\int_{\Sigma_t}\tau dA_t&\le&c\int_0^t\left(\int_{\Sigma_s}H_s\phi dA_s\right)ds\\
&\le&c\int_0^t\theta(s)\left(\int_{\Sigma_s}\phi dA_s\right)ds,
\end{eqnarray*} 
i.e.,
\begin{eqnarray*}
\theta'(t)\int_{\Sigma_t}\frac{1}{\phi}dA_t\le\int_0^t\theta(s)\left(c\int_{\Sigma_s}\phi dA_s\right)ds+\theta(t)\int_{\Sigma_t}\tau dA_t,\,\,\,\forall t\in[0,\varepsilon).
\end{eqnarray*} 
It follows from Lemma \ref{lemma.Calculus.Lemma} that $\theta(t)=0$, i.e., $\Sigma_t$ is a MOTS for each $t\in[0,\varepsilon)$, because $\theta(t)\ge0$ since $\Sigma$ is weakly outermost. 

Since $\theta(t)=0$ for each $t\in[0,\varepsilon)$, all inequalities above must be equalities. Then, $Y=X-\nabla\ln\phi=0$, $\chi=0$, and $\mu+J(\nu)=\mu-|J|=-c$ on $U\approx[0,\varepsilon)\times\Sigma$. Moreover, $A(\Sigma_t)=A(\Sigma_0)=4\pi(g(\Sigma)-1)/c$ for all $t\in[0,\varepsilon)$. Because $\theta'(t)=0$, it follows from (\ref{eq.First.Variation.Theta}) and Lemma \ref{lemma.AnderssonMarsSimon} that $\Sigma_t$ is stable. Thus, using Proposition \ref{prop.GibbonsWoolgar}, $\Sigma_t$ has constant Gaussian curvature $\kappa_t=-c$, for each $t\in[0,\varepsilon)$.

Now, because $t\longmapsto A(\Sigma_t)$ is constant for $t\in[0,\varepsilon)$,
\begin{eqnarray*}
 0=\frac{d}{dt}A(\Sigma_t)=\int_{\Sigma_t}H_t\phi dA_t. 
\end{eqnarray*}
This implies $H_t=0$ for each $t\in[0,\varepsilon)$, since $0=\theta(t)\ge H_t$ and $\phi=\phi_t>0$. Then, $\Sigma_t$ is a minimal MOTS, which implies $\tr_{\Sigma_t}K=0$ for each $t\in[0,\varepsilon)$. In this case, the null mean curvature $\theta_-(t)=\tr_{\Sigma_t}K-H_t$ of $\Sigma_t$ with respect to $l_-(t)=u-\nu_t$ also vanishes everywhere. Applying (\ref{eq.First.Variation.Theta}) for $\theta_-$ and $\phi_-=-\phi$ instead of $\theta=\theta_+$ and $\phi$, respectively, we have
\begin{eqnarray}\label{eq1.Teor.1.4.3}
0=\theta_-'=-\Delta\phi_-+2\langle X_-,\nabla\phi_-\rangle+(Q_--|X_-|^2+\div X_-)\phi_-,
\end{eqnarray}
where
\begin{eqnarray}\label{eq2.Teor.1.4.3}
 Q_-&=&\kappa-(\mu+J(-\nu))-\frac{1}{2}|\chi_-|^2\nonumber\\
 &=&-c-(\mu+|J|)-\frac{1}{2}|\chi_-|^2\nonumber\\
 &=&-2|J|-\frac{1}{2}|\chi_-|^2,\\
 X_-&=&-X=-\frac{1}{\phi}\nabla\phi.\label{eq3.Teor.1.4.3}
\end{eqnarray}
Substituting (\ref{eq2.Teor.1.4.3}) and (\ref{eq3.Teor.1.4.3}) into (\ref{eq1.Teor.1.4.3}) (and remembering that $\phi_-=-\phi$), we get
\begin{eqnarray*}
\Delta\phi+\frac{|\nabla\phi|^2}{\phi}+\left(|J|+\frac{1}{4}|\chi_-|^2\right)\phi=0,
\end{eqnarray*}
which, after integration over $\Sigma_t$, implies
\begin{eqnarray*}
|\nabla\phi|=|\chi_-|=|J|=0\,\,\,\mbox{on}\,\,\,U.
\end{eqnarray*}

Equation (\ref{eq.chi_pm}) now implies that $K|_{T_x\Sigma_t\times T_x\Sigma_t}=0$ and that $(\Sigma,g_t)$ is totally geodesic in $M$ for each $t\in[0,\varepsilon)$, where $g_t$ is the Riemannian metric on $\Sigma_t\approx\{t\}\times\Sigma$ induced from $(M,g)$. Writing $g=\phi^2dt^2+g_t$ on $U\approx[0,\varepsilon)\times\Sigma$ and observing that $\phi=\phi_t$ depends only on $t\in[0,\varepsilon)$, it follows that $g_t$ does not depend on $t\in[0,\varepsilon)$. Then, after the simple change of variable $ds=\phi(t)dt$, $\phi(t)=\phi_t$, we can see that $g$ has the product structure $ds^2+g_0$ on $U$, where $(\Sigma,g_0)$ has constant Gaussian curvature $\kappa=-c$.
\end{proof}

\section{The Case of High Dimension with Negative $\sigma$-Constant}\label{sec.HighDimension}

In this section, we extend the main result of last section to the case of high-dimensional MOTSs with negative $\sigma$-constant. But, before stating our result, we are going to present some terminologies.

Let $\Sigma^n$ be a connected closed (compact with no boundary) $n$-manifold, $n\ge3$. Denote by $\mathcal{M}(\Sigma)$ the set of all Riemannian metrics on $\Sigma$. The {\em Einstein-Hilbert functional} $\mathcal{E}:\mathcal{M}(\Sigma)\to\R$ is defined by 
\begin{eqnarray*}
 \E(g)=\dfrac{\int_\Sigma R_gdv_g}{\Vol(\Sigma^n,g)^{\frac{n-2}{n}}},
\end{eqnarray*} 
where $R_g$ is the scalar curvature of $(\Sigma,g)$. Denote by $[g]=\{e^{2f}g;f\in C^\infty(\Sigma)\}$ the conformal class of $g\in\M(\Sigma)$. The {\em Yamabe invariant} of $(\Sigma,[g])$ is defined as the following conformal invariant:
\begin{eqnarray*}
 \Y(\Sigma,[g])=\inf_{\tilde g\in[g]}\E(\tilde g).
\end{eqnarray*}

The classical solution of the Yamabe problem by Yamabe \cite{Yamabe}, Trudinger \cite{Trudinger}, Aubin \cite{Aubin76} (se also \cite{Aubin98}), and Schoen \cite{Schoen84} says that every conformal class $[g]$ contains metrics $\hat g$, called {\em Yamabe metrics}, which realize the minimum:
\begin{eqnarray*}
 \E(\hat g)=\Y(\Sigma,[g]).
\end{eqnarray*} 
Such metrics have constant scalar curvature given by 
\begin{eqnarray*}
 R_{\hat g}=\Y(\Sigma^n,[g])\Vol(\Sigma^n,\hat g)^{-\frac{2}{n}}. 
\end{eqnarray*}
Furthermore,
\begin{eqnarray*}
 \Y(\Sigma^n,[g])\le\Y(\s^n,[g_{\can}]),
\end{eqnarray*}
and equality holds if and only if $(\Sigma^n,g)$ is conformally diffeomorphic to the Euclidean $n$-sphere $\s^n\subset\R^{n+1}$ endued with the canonical metric $g_{\can}$. 

The following topological invariant was introduced by Schoen in lectures given in 1987, which were published two years later  \cite{Schoen1989} (see also Kobayashi \cite{Kobayashi}),
\begin{eqnarray*}
\sigma(\Sigma^n)=\sup_{g\in\M(\Sigma^n)}\Y(\Sigma^n,[g]).
\end{eqnarray*}
This invariant is called the {\em $\sigma$-constant} of $\Sigma$.

\begin{remark}
Using the solution of the Yamabe problem, it is not difficult to prove that $\sigma(\Sigma)>0$ if and only if $\Sigma$ admits a Riemannian metric of positive scalar curvature. On the other hand, observe that if $S^2$ is an orientable closed surface, then its Euler-Poincaré characteristic $\chi(S)=2(1-g(S))>0$ if and only if $g(S)=0$, i.e., $S^2$ is topologically $\s^2$, which by Gauss-Bonnet and uniformization theorems is equivalent to say that $S^2$ admits a Riemannian metric of positive Gaussian curvature. Having this in mind, the $\sigma$-constant in dimension $n\ge3$ shares the same property with the Euler-Poincaré characteristic in dimension 2.
\end{remark}

Now, let us present the first result of this section. The following volume estimate was obtained by Galloway and Murchadha \cite{GallowayMurchadha} (see also \cite{CaiGalloway2001} for the time-symmetric case). Our contribution consists in the {\em infinitesimal rigidity} obtained when equality occurs, which is based on \cite{GallowayMendes} and \cite{Moraru}.

\begin{proposition}\label{prop.2.5.3}
Let $\Sigma^n$ be a stable closed MOTS with $\sigma(\Sigma)<0$ in an initial data set $M^{n+1}=(M^{n+1},g,K)$, $n\ge3$. Suppose that $\mu+J(\nu)\ge-c$ on $\Sigma$ for some constant $c>0$. Then, the volume of $\Sigma$ satisfies
\begin{eqnarray}\label{eq.2.24}
\Vol(\Sigma^n)\ge\left(\frac{|\sigma(\Sigma^n)|}{2c}\right)^{\frac{n}{2}}.
\end{eqnarray}   
Furthermore, if equality holds, the metric $g_\Sigma$ on $\Sigma$ induced from $M$ is Einstein and has constant scalar curvature $R_\Sigma=-2c$, the null second fundamental form $\chi=\chi_+$ of $\Sigma$ vanishes identically, $\mu+J(\nu)=-c$ on $\Sigma$, and $\lambda_1(L_0)=\lambda_1(L)=0$.
\end{proposition}

\begin{proof}
Since $\Sigma$ is stable, $\lambda_1(L)\ge0$, where $L$ is the MOTSs stability operator. Then, by (\ref{eq.Stability.L0}), 
 \begin{eqnarray}
  0&\le&2\int_\Sigma(|\nabla u|^2+Qu^2)dv\label{eq.aux.4}\\
  &=&\int_\Sigma(2|\nabla u|^2+(R_\Sigma-2(\mu+J(\nu))-|\chi|^2)u^2)dv\nonumber\\
  &\le&\int_\Sigma(2|\nabla u|^2+(R_\Sigma+2c)u^2)dv,\label{eq.aux.4.1}
 \end{eqnarray}
 for all $u\in C^\infty(\Sigma)$, $u>0$. Using $2<\frac{4(n-1)}{n-2}$ and Hölder inequality into (\ref{eq.aux.4.1}), we have   
 \begin{eqnarray}
  0&\le&\int_\Sigma\left(\frac{4(n-1)}{n-2}|\nabla u|^2+R_\Sigma u^2\right)dv+2c\int_\Sigma u^2dv\label{eq.aux.5}\\
  &\le&\int_\Sigma\left(\frac{4(n-1)}{n-2}|\nabla u|^2+R_\Sigma u^2\right)dv
  +2c\Vol(\Sigma)^{\frac{2}{n}}\left(\int_\Sigma u^\frac{2n}{n-2}dv\right)^{\frac{n-2}{n}}\label{eq.aux.6}
  \end{eqnarray}
  for all $u\in C^\infty(\Sigma)$, $u>0$. On the other hand, it is well known that for $g=u^{\frac{4}{n-2}}g_\Sigma$, $u>0$, we have
  \begin{eqnarray}\label{eq.R_g}
  R_gu^{\frac{n+2}{n-2}}=-\frac{4(n-1)}{n-2}\Delta u+R_\Sigma u,
  \end{eqnarray} 
  which implies
  \begin{eqnarray}\label{eq.aux.6.1}
  \E(u^{\frac{4}{n-2}}g_\Sigma)=\frac{\int_\Sigma(\frac{4(n-1)}{n-2}|\nabla u|^2+R_\Sigma u^2)dv}{\left(\int_\Sigma u^{\frac{2n}{n-2}}dv\right)^{\frac{n-2}{n}}}.
  \end{eqnarray}
 Therefore, using (\ref{eq.aux.6.1}) into (\ref{eq.aux.6}), 
 \begin{eqnarray}\label{eq.aux.7}
  0\le\E(u^{\frac{4}{n-2}}g_\Sigma)+2c\Vol(\Sigma)^{\frac{2}{n}},
 \end{eqnarray} 
 for all $u\in C^\infty(\Sigma)$, $u>0$, which implies 
 \begin{eqnarray}\label{eq.aux.8}
  0&\le&\inf_{u\in C^\infty(\Sigma),\, u>0}\E(u^{\frac{4}{n-2}}g_\Sigma)+2c\Vol(\Sigma)^{\frac{2}{n}}\\
  &=&\Y(\Sigma,[g_\Sigma])+2c\Vol(\Sigma)^{\frac{2}{n}}\nonumber\\
  &\le&\sigma(\Sigma)+2c\Vol(\Sigma)^{\frac{2}{n}}.\label{eq.aux.8.1}  
 \end{eqnarray} 
This finishes the proof of (\ref{eq.2.24}).

Now, suppose that equality in (\ref{eq.2.24}) holds. Then, inequalities (\ref{eq.aux.8}) and (\ref{eq.aux.8.1}) must be equalities.
Choose $u_0\in C^\infty(\Sigma)$, $u_0>0$, such that $\hat g=u_0^{4/(n-2)}g_\Sigma$ is a Yamabe metric. Then, inequalities (\ref{eq.aux.4}), (\ref{eq.aux.4.1}), (\ref{eq.aux.5}), (\ref{eq.aux.6}), and (\ref{eq.aux.7}) must be equalities when $u=u_0$. Therefore, $\chi=0$, $\mu+J(\nu)=-c$ on $\Sigma$, and $|\nabla u_0|$ vanishes identically, since $2<\frac{4(n-1)}{n-2}.$ This implies that $u_0$ is constant (in particular, $g_\Sigma$ is also a Yamabe metric) and $\int_\Sigma Qdv=0$. Proceeding analogously to the proof of Proposition \ref{prop.GibbonsWoolgar}, we have $R_\Sigma=-2c$ and $\lambda_1(L_0)=\lambda_1(L)=0$. Equality in (\ref{eq.aux.8.1}) implies that $g_\Sigma$ realizes $\sigma(\Sigma)$, i.e., $\E(g_\Sigma)=\Y(\Sigma,[g_\Sigma])=\sigma(\Sigma)$. Then, by \cite[pp. 126-127]{Schoen1989}, $g_\Sigma$ is Einstein.
\end{proof}

In the following, we are going to present the main result of this section. But, before doing that, let us make an important remark.

\begin{remark}\label{remark.Uniquiness.Yamabe.Metrics}
Let $\Sigma^n$ be a closed manifold of dimension $n\ge3$. Suppose that $g\in\M(\Sigma)$ satisfies $\Y(\Sigma^n,[g])<0$. Then, it is well known that given $c>0$, there exists a unique metric $\tilde g\in[g]$ such that $R_{\tilde g}=-2c$. In fact, the existence follows from the solution of the Yamabe problem. To prove the uniqueness, let $g_1,g_2\in[g]$ be such that $R_{g_1}=-2c=R_{g_2}$. Since $g_1$ and $g_2$ are conformal, we have $g_2=u^{4/(n-2)}g_1$ for some $u\in C^\infty(\Sigma)$, $u>0$. Then, using (\ref{eq.R_g}) for $g_1$ and $g_2$, it follows that
\begin{eqnarray*}
-2cu^{\frac{n+2}{n-2}}=-\frac{4(n-1)}{n-2}\Delta_{g_1}u-2cu,
\end{eqnarray*}
i.e.,
\begin{eqnarray*}
2cu^{\frac{n+2}{n-2}}=\frac{4(n-1)}{n-2}\Delta_{g_1}u+2cu.
\end{eqnarray*}
Then, taking $x_0\in\Sigma$ such that $u(x_0)=\min u$,  
\begin{eqnarray*}
2c(u(x_0))^{\frac{n+2}{n-2}}=\frac{4(n-1)}{n-2}\Delta_{g_1}u(x_0)+2cu(x_0)\ge 2cu(x_0),
\end{eqnarray*}
which implies $u(x_0)\ge1$. In turn, taking $x_1\in\Sigma$ such that $u(x_1)=\max u$, we have $u(x_1)\le 1$. Thus, $u\equiv 1$.
\end{remark}

The main result of this section is the following. Its proof is based on \cite{GallowayMendes} and \cite{Moraru}.

\begin{theorem}\label{theo.Negative.sigma.Constant}
Let $M^{n+1}=(M^{n+1},g,K)$ be a $(n+1)$-dimensional initial data set, $n\ge3$. Let $\Sigma^n$ be a weakly outermost closed MOTS in $M^{n+1}$ with $\sigma$-constant $\sigma(\Sigma)<0$. Suppose that $\mu-|J|\ge-c$ for some constant $c>0$ and that $K$ is $n$-convex, both on $M_+$. Then, if $\Vol(\Sigma^n)=(|\sigma(\Sigma^n)|/2c)^{n/2}$, the following hold.
\begin{enumerate}
\item An outer neighborhood $U\approx[0,\varepsilon)\times\Sigma$ of $\Sigma$ in $M$ is isometric to 
\begin{eqnarray*}
([0,\varepsilon)\times\Sigma,dt^2+g_\Sigma),
\end{eqnarray*}
and $(\Sigma,g_\Sigma)$ is Einstein with constant scalar curvature $R_\Sigma=-2c$.
\item Each slice $\Sigma_t\approx\{t\}\times\Sigma$ is totally geodesic as a submanifold of spacetime. Equivalently, $\chi_+(t)=\chi_-(t)=0$, where $\chi_\pm(t)$ are the null second fundamental forms of $\Sigma_t$.
\item $K(\cdot,\cdot)|_{T_x\Sigma_t}=0$ and $K(\nu_t,\cdot)|_{T_x\Sigma_t}=0$ for each $x\in\Sigma_t$, where $\nu_t$ is the outer unit normal to $\Sigma_t$, and $J=0$ on $U$.
\end{enumerate}
\end{theorem}

\begin{proof}
Since weakly outermost closed MOTSs are stable (Remark \ref{remark.WeaklyOutermostStable}), by Proposition \ref{prop.2.5.3} we have $\lambda(L)=0$, where $L$ is the MOTSs stability operator. Then, we can take $\{\Sigma_t\}_{t\in(-\varepsilon,\varepsilon)}$, $\theta=\theta(t)$, and $\phi=\phi_t$ as in Lemma \ref{lemma.Folliation}. By (\ref{eq.First.Variation.Theta}),
\begin{eqnarray}
  \theta'&=&-\Delta\phi+2\langle X,\nabla\phi\rangle+\left(Q-|X|^2+\div X-\frac{1}{2}\theta^2+\theta\tau\right)\phi\nonumber\\
  &=&\left(-|Y|^2+\div Y+Q-\frac{1}{2}\theta^2+\theta\tau\right)\phi\nonumber\\
  &\le&\left(-|Y|^2+\div Y+Q+\theta\tau\right)\phi,\label{eq.aux.9}
 \end{eqnarray}
 where $Y=X-\nabla\ln\phi$. Above, $\Delta=\Delta_t$, $\langle\,,\,\rangle=\langle\,,\,\rangle_t$, $X=X_t$, $\nabla=\nabla_t$ (gradient operator), $Q=Q_t$, and $\div=\div_t$ are the respective entities associated to $\Sigma_t$, for each $t\in(-\varepsilon,\varepsilon)$. It follows from (\ref{eq.aux.9}) that 
 \begin{eqnarray*}
  \theta'(t)\frac{u^2}{\phi}-\theta(t)\tau u^2&\le&-u^2|Y|^2+u^2\div Y+Qu^2\\
  &=&-u^2|Y|^2-2u\langle\nabla u,Y\rangle+\div(u^2Y)\\
  &&+\left(\frac{1}{2}R_{\Sigma_t}-(\mu+J(\nu))-\frac{1}{2}|\chi|^2\right)u^2\\
  &\le&-u^2|Y|^2+2|u||\nabla u||Y|+\div(u^2Y)\\
  &&+\left(\frac{1}{2}R_{\Sigma_t}-(\mu-|J|)-\frac{1}{2}|\chi|^2\right)u^2\\
  &\le&|\nabla u|^2+\div(u^2Y)+\left(\frac{1}{2}R_{\Sigma_t}+c\right)u^2,
 \end{eqnarray*}
for all $u\in C^\infty(\Sigma_t)$, $u>0$, and $t\in[0,\varepsilon)$. Above we have used that $\mu-|J|\ge-c$ on $M_+$. Integrating over $\Sigma_t$, observing that $\theta'(t)$ is also constant on $\Sigma_t$, and using Hölder inequality, we get
 \begin{eqnarray*}
  2\left(\theta'(t)\int_{\Sigma_t}\frac{u^2}{\phi}dv_t-\theta(t)\int_{\Sigma_t}\tau u^2dv_t\right)\!\!\!&\le&\!\!\!
  \int_{\Sigma_t}(2|\nabla u|^2+R_{\Sigma_t}u^2)dv_t+2c\int_{\Sigma_t}u^2dv_t\\
  \!\!\!&\le&\!\!\!\int_{\Sigma_t}\left(\frac{4(n-1)}{n-2}|\nabla u|^2+R_{\Sigma_t}u^2\right)dv_t\\
  \!\!\!&&\!\!\!+2c\Vol(\Sigma_t)^{\frac{2}{n}}\left(\int_{\Sigma_t} u^\frac{2n}{n-2}dv_t\right)^{\frac{n-2}{n}},
 \end{eqnarray*} 
for all $u\in C^\infty(\Sigma_t)$, $u>0$, and $t\in[0,\varepsilon)$. It follows from Remark \ref{remark.Uniquiness.Yamabe.Metrics} that for each $t\in[0,\varepsilon)$, there exists a unique $u_t\in C^\infty(\Sigma_t)$, $u_t>0$, such that $\hat g_t=u_t^{4/(n-2)}g_{\Sigma_t}$ is a Yamabe metric with constant scalar curvature $R_{\hat g_t}=-2c$, given that 
\begin{eqnarray*}
\Y(\Sigma_t,[g_{\Sigma_t}])\le\sigma(\Sigma)<0,
\end{eqnarray*}
where $g_{\Sigma_t}$ is the Riemannian metric on $\Sigma_t$ induced from $M$. Therefore,
\begin{eqnarray*}
  \frac{\ds 2\left(\theta'(t)\int_{\Sigma_t}u_t^2\phi^{-1}dv_t-\theta(t)\int_{\Sigma_t}\tau u_t^2dv_t\right)}{\ds\left(\int_{\Sigma_t}
  u_t^\frac{2n}{n-2}dv_t\right)^{\frac{n-2}{n}}}&\le&\E(u_t^{\frac{4}{n-2}}g_{\Sigma_t})+2c\Vol(\Sigma_t)^{\frac{2}{n}}\\
  &=&\Y(\Sigma_t,[g_{\Sigma_t}])+2c\Vol(\Sigma_t)^{\frac{2}{n}}\\
  &\le&\sigma(\Sigma)+2c\Vol(\Sigma_t)^{\frac{2}{n}}\\
  &=&\left(-2c\Vol(\Sigma)^{\frac{2}{n}}+2c\Vol(\Sigma_t)^{\frac{2}{n}}\right)\\
  &=&\frac{4c}{n}\int_0^t\Vol(\Sigma_s)^{\frac{2-n}{n}}\frac{d}{ds}\Vol(\Sigma_s)ds,
\end{eqnarray*}
for each $t\in[0,\varepsilon)$. Using that $\theta(t)=\tr_{\Sigma_t}K+H_t\ge H_t$ for $t\in[0,\varepsilon)$, since $K$ is $n$-convex on $M_+$, together with the first variation of area (volume), we have 
\begin{eqnarray*}
\frac{\ds\theta'(t)\int_{\Sigma_t}u_t^2\phi^{-1}dv_t-\theta(t)\int_{\Sigma_t}\tau u_t^2dv_t}{\ds\left(\int_{\Sigma_t}
  u_t^\frac{2n}{n-2}dv_t\right)^{\frac{n-2}{n}}}&\le&\frac{2c}{n}\int_0^t\Vol(\Sigma_s)^{\frac{2-n}{n}}\left(\int_{\Sigma_s}H_s\phi dv_s\right)ds\\
  &\le&\frac{2c}{n}\int_0^t\theta(s)\left(\Vol(\Sigma_s)^{\frac{2-n}{n}}\int_{\Sigma_s}\phi dv_s\right)ds.
\end{eqnarray*}
Now, observing that by the uniqueness of $\hat g_t$, $u_t$ is continuous on $t\in[0,\varepsilon)$, we can use Lemma \ref{lemma.Calculus.Lemma} for
\begin{eqnarray*}
  &\ds f(t)=\theta(t),\,\,\,\eta(t)=\left(\int_{\Sigma_t}u_t^2\phi^{-1}dv_t\right)
  \left(\int_{\Sigma_t}u_t^\frac{2n}{n-2}dv_t\right)^{-\frac{n-2}{n}},&\\
  &\ds\rho(t)=\left(\int_{\Sigma_t}\tau u_t^2dv_t\right)\left(\int_{\Sigma_t}u_t^\frac{2n}{n-2}dv_t\right)^{-\frac{n-2}{n}},\,\,\,\mbox{and}\,\,\,
  \xi(t)=\frac{2c}{n}\Vol(\Sigma_t)^{\frac{2-n}{n}}\int_{\Sigma_t}\phi dv_t.&
 \end{eqnarray*} 
Thus, $\theta(t)=0$ for all $t\in[0,\varepsilon)$. Then, continuing as in the proof of Theorem \ref{theo.Split.Dimension.3}, we have the result.
\end{proof}

\bibliographystyle{amsplain}
\bibliography{bib.bib}

\providecommand{\bysame}{\leavevmode\hbox to3em{\hrulefill}\thinspace}
\providecommand{\MR}{\relax\ifhmode\unskip\space\fi MR }
\providecommand{\MRhref}[2]{%
  \href{http://www.ams.org/mathscinet-getitem?mr=#1}{#2}
}
\providecommand{\href}[2]{#2}
\begin{thebibliography}{10}

\bibitem{Ambrozio}
L.~C. Ambrozio, \emph{Rigidity of area-minimizing free boundary surfaces in
  mean convex three-manifolds}, J. Geom. Anal. \textbf{25} (2015), no.~2,
  1001--1017. \MR{3319958}

\bibitem{AnderssonMarsSimon2005}
L.~Andersson, M.~Mars, and W.~Simon, \emph{Local existence of dynamical and
  trapping horizons}, Phys. Rev. Lett. \textbf{95} (2005), 111102.

\bibitem{AnderssonMarsSimon2008}
\bysame, \emph{Stability of marginally outer trapped surfaces and existence of
  marginally outer trapped tubes}, Adv. Theor. Math. Phys. \textbf{12} (2008),
  no.~4, 853--888. \MR{2420905}

\bibitem{Aubin76}
T.~Aubin, \emph{\'{E}quations diff\'erentielles non lin\'eaires et probl\`eme
  de {Y}amabe concernant la courbure scalaire}, J. Math. Pures Appl. (9)
  \textbf{55} (1976), no.~3, 269--296. \MR{0431287 (55 \#4288)}

\bibitem{Aubin98}
\bysame, \emph{Some nonlinear problems in {R}iemannian geometry}, Springer
  Monographs in Mathematics, Springer-Verlag, Berlin, 1998. \MR{1636569
  (99i:58001)}

\bibitem{BrayBrendleNeves}
H.~Bray, S.~Brendle, and A.~Neves, \emph{Rigidity of area-minimizing
  two-spheres in three-manifolds}, Comm. Anal. Geom. \textbf{18} (2010), no.~4,
  821--830. \MR{2765731}

\bibitem{Cai}
M.~Cai, \emph{Volume minimizing hypersurfaces in manifolds of nonnegative
  scalar curvature}, Adv. Stud. Pure Math. \textbf{34} (2002), 1--7.

\bibitem{CaiGalloway2000}
M.~Cai and G.~J. Galloway, \emph{Rigidity of area minimizing tori in
  3-manifolds of nonnegative scalar curvature}, Comm. Anal. Geom. \textbf{8}
  (2000), no.~3, 565--573. \MR{1775139 (2001j:53051)}

\bibitem{CaiGalloway2001}
\bysame, \emph{On the topology and area of higher-dimensional black holes},
  Classical Quantum Gravity \textbf{18} (2001), no.~14, 2707--2718.
  \MR{1846368}

\bibitem{Fischer-ColbrieSchoen}
D.~Fischer-Colbrie and R.~Schoen, \emph{The structure of complete stable
  minimal surfaces in {$3$}-manifolds of nonnegative scalar curvature}, Comm.
  Pure Appl. Math. \textbf{33} (1980), no.~2, 199--211. \MR{562550 (81i:53044)}

\bibitem{Galloway2008}
G.~J. Galloway, \emph{Rigidity of marginally trapped surfaces and the topology
  of black holes}, Comm. Anal. Geom. \textbf{16} (2008), no.~1, 217--229.
  \MR{2411473 (2009e:53087)}

\bibitem{Galloway2011}
\bysame, \emph{Stability and rigidity of extremal surfaces in {R}iemannian
  geometry and general relativity}, Surveys in geometric analysis and
  relativity, Adv. Lect. Math. (ALM), vol.~20, Int. Press, Somerville, MA,
  2011, pp.~221--239. \MR{2906927}

\bibitem{GallowayMendes}
G.~J. Galloway and A.~Mendes, \emph{Rigidity of marginally outer trapped
  2-spheres}, \href{http://arxiv.org/abs/1506.00611v3}{arXiv:1506.00611v3}
  (2016), to appear in Comm. Anal. Geom.

\bibitem{GallowayMurchadha}
G.~J. Galloway and N.~Ó. Murchadha, \emph{Some remarks on the size of bodies
  and black holes}, Class. Quantum Grav. \textbf{25} (2008), no.~10, 105009.

\bibitem{GallowaySchoen}
G.~J. Galloway and R.~Schoen, \emph{A generalization of {H}awking's black hole
  topology theorem to higher dimensions}, Comm. Math. Phys. \textbf{266}
  (2006), no.~2, 571--576. \MR{2238889 (2007i:53078)}

\bibitem{Gibbons}
G.~W. Gibbons, \emph{Some comments on gravitational entropy and the inverse
  mean curvature flow}, Class. Quantum Grav. \textbf{16} (1999), no.~6,
  1677–1687.

\bibitem{KazdanWarner}
J.~L. Kazdan and F.~W. Warner, \emph{Prescribing curvatures}, Differential
  geometry ({P}roc. {S}ympos. {P}ure {M}ath., {V}ol. {XXVII}, {S}tanford
  {U}niv., {S}tanford, {C}alif., 1973), {P}art 2, Amer. Math. Soc., Providence,
  R.I., 1975, pp.~309--319. \MR{0394505 (52 \#15306)}

\bibitem{Kobayashi}
O.~Kobayashi, \emph{Scalar curvature of a metric with unit volume},
  Mathematische Annalen \textbf{279} (1987), no.~2, 253--265.

\bibitem{MicallefMoraru}
M.~Micallef and V.~Moraru, \emph{Splitting of 3-manifolds and rigidity of
  area-minimising surfaces}, Proc. Amer. Math. Soc. \textbf{143} (2015), no.~7,
  2865--2872. \MR{3336611}

\bibitem{Moraru}
V.~Moraru, \emph{On area comparison and rigidity involving the scalar
  curvature}, J. Geom. Anal. \textbf{26} (2016), no.~1, 294--312. \MR{3441515}

\bibitem{Nunes}
I.~Nunes, \emph{Rigidity of area-minimizing hyperbolic surfaces in
  three-manifolds}, J. Geom. Anal. \textbf{23} (2013), no.~3, 1290--1302.
  \MR{3078354}

\bibitem{Schoen84}
R.~Schoen, \emph{Conformal deformation of a {R}iemannian metric to constant
  scalar curvature}, J. Differential Geom. \textbf{20} (1984), no.~2, 479--495.
  \MR{788292 (86i:58137)}

\bibitem{Schoen1989}
\bysame, \emph{Topics in calculus of variations: Lectures given at the 2nd 1987
  session of the centro internazionale matematico estivo (c.i.m.e.) held at
  montecatini terme, italy, july 20--28, 1987}, ch.~Variational theory for the
  total scalar curvature functional for {R}iemannian metrics and related
  topics, pp.~120--154, Springer Berlin Heidelberg, Berlin, Heidelberg, 1989.

\bibitem{SchoenYau1979}
R.~Schoen and S.~T. Yau, \emph{Existence of incompressible minimal surfaces and
  the topology of three-dimensional manifolds with nonnegative scalar
  curvature}, Ann. of Math. (2) \textbf{110} (1979), no.~1, 127--142.
  \MR{541332 (81k:58029)}

\bibitem{Trudinger}
N.~S. Trudinger, \emph{Remarks concerning the conformal deformation of
  {R}iemannian structures on compact manifolds}, Ann. Scuola Norm. Sup. Pisa
  (3) \textbf{22} (1968), 265--274. \MR{0240748 (39 \#2093)}

\bibitem{Woolgar}
E.~Woolgar, \emph{Bounded area theorems for higher-genus black holes}, Class.
  Quantum Grav. \textbf{16} (1999), no.~9, 3005–3012.

\bibitem{Yamabe}
H.~Yamabe, \emph{On a deformation of {R}iemannian structures on compact
  manifolds}, Osaka Math. J. \textbf{12} (1960), 21--37. \MR{0125546 (23
  \#A2847)}

\end{thebibliography}

\end{document}